\documentclass[12pt,a4paper]{article}
\usepackage{amssymb,amsmath}
\usepackage{amsfonts}
\usepackage{amsthm}
\usepackage[T1]{fontenc}
\usepackage[ansinew,latin1]{inputenc}
\usepackage[english]{babel}
\usepackage{graphicx}
\usepackage{makeidx}
\theoremstyle{plain}
\newtheorem{thm}{Theorem}
\newtheorem{lem}[thm]{Lemma}

\newtheorem{alg}[thm]{Algorithm}
\theoremstyle{definition}
\newtheorem{defi}[thm]{Definition}

\theoremstyle{remark}

\begin{document}

\title{How to use the Fast Fourier Transform in Large Finite Fields}

\author{Petur Birgir Petersen}

\date{Rungsted Kyst 2009}

\maketitle

\begin{abstract}

The article contents suggestions on how to perform the Fast Fourier
Transform over Large Finite Fields. The technique is to use the fact
that the multiplicative groups of specific prime fields are
surprisingly composite. \footnote{MSC 11, 42, 68 and 94.
}

\end{abstract}

\section{Introduction}
In 2003 Gao published the article \emph{A New Algorithm for Decoding
Reed -- Solomon Codes} \cite{Gao}. Gaos algorithm can be executed
through out with the use of the Discrete Fourier Transform (DFT) in
a Finite Field. The algorithm will of course be much faster using
the Fast Fourier Transformation (FFT). The coding and decoding of
Reed--Solomon Codes is often performed over Finite Fields
$\mathbb{F}_{q}$ of order $q=2^m$, where $m \in \mathbb{N}$. In 2006
and 2007 Truong, Chen, Wang, Chang \& Reed published the article
\cite{Truong}, \cite{Truong er} : \emph{Fast prime factor, discrete
Fourier algorithms over $GF(2^m)$, for $8 \leq m \leq 10$}, which is
is a sort of follow up on another article \cite{Reed} that treats
the cases $n=4, 5, 6, 8$. These results are very important, but
$2^{10} = 1024$, and for instance digital TV signals use much bigger
files, say in the order of $2^{20} \approx 10^6$. For large files
like this my suggestion is to augment the file with an extra bit,
and let the FFT be performed over a prime field $\mathbb{F}_{p}$.
For certain primes this can be performed efficient by an algorithm
based on the Cooley--Tukey algorithm \cite{Cooley} from 1965. In the
next section the algorithm will be explained, and the last section
contents a variety of suggestions on well--suited primes $p$, where
the FFT over $\mathbb{F}_{p}$ will be specially efficient.

\section{Fast Fourier Transformation over Finite Fields}

\begin{defi} \label{DFT}
Let $\omega$ be an element in $\mathbb{F}_{p^{m}}$ of order $n$
where $n \mid p^{m}-1$. The Discrete Fourier Transform (DFT) of the
$n$-tuple $ \underline{v} = (v_{0}, v_{1}, ..., v_{n-1}) \in
\mathbb{F}_{p^{m}}^{n} $ is the $n$-tuple $\underline{V}$ with
components given by

\begin{equation}\label{V}
    V_{j}=\sum_{i=0}^{n-1} {\omega^{ij} v_{i}}, \quad  j = 0,1, ..., n-1
\end{equation}
\end{defi}

The Inverse Discrete Fourier Transform (IDFT) of the $n$-tuple $
\underline{V}\in \mathbb{F}_{p^{m}}^{n} $ is the $n$-tuple

\begin{equation*}
    (v_{i})=((n^{-1} \mathrm{mod} \ p) \sum_{i=0}^{n-1} {\omega^{-ij} V_{j}}), \quad  i = 0,1, ..., n-1
\end{equation*}

For a proof see e.g. \cite{Blahut}. Notice that the IDFT apart from
the factor $(n^{-1} \mathrm{mod} \ p)$ also is a DFT.

Now assume that $n \mid p^{m}-1$ is composite: $n=r_{1} r_{2}$. The
indices in definition \ref{DFT} can be rewritten like this:
\[ j= j_{1}r_{1}+j_{0}, \quad j_{0}=0,1,\ldots, r_{1}-1, \quad , j_{1} =
0,1, \ldots , r_{2}-1 \]

\[ i= i_{1}r_{2}+i_{0}, \quad i_{0}=0,1,\ldots, r_{2}-1, \quad ,
i_{1} = 0,1, \ldots , r_{1}-1 \].

Replacing $v_{i}$ by $x_{0}(i)$,  equation ( \ref{V}) now can be
rewritten as:

\[
V(j_{1},j_{0})= \sum_{i_{0}=0}^{r_{2}-1}(\sum_{i_{1}}^{r_{1}-1}
x_{0}(i_{1},i_{0}) \omega^{i_{1} r_{2} j} )\omega^{i_{0}j}
\]

Since $\omega^{n} = \omega^{r_{1} r_{2}} =1$, then
$\omega^{i_{1}r_{2}j} = \omega^{i_{1}r_{2}j_{0}}$.

Set
\[
x_{1}(j_{0},i_{0})=\sum_{i_{1}=0}^{r_{1}-1} x_{0}(i_{1},i_{0})
\omega^{j_{0} i_{1} r_{2} }
\]

Then

\[
V(j_{1},j_{0})= \sum_{i_{0}=0}^{r_{2}-1} x_{1}(j_{0},i_{0})
\omega^{i_{1} r_{2} j} )\omega^{(j_{1}r_{1}+j_{0})i_{0}}
\]
It will require $nr_{1}$ multiplications and $n(r_{1}-1)$ additions
to calculate $x_{1}$ for all $(j_0, i_0)$ and $nr_{2}$
multiplications and $n(r_{2}-1)$ additions to calculate
\underline{V} from $x_1$. This will give a total of $n(r_{1}+r_{2})$
multiplications and $n(r_{1}+r_{2} - 2)$ additions in
$\mathbb{F}_{p^m}$. For $n\geq 4$ this is faster than the DFT which
requires $n^2$ multiplications and $n(n-1)$ additions in
$\mathbb{F}_{p^m}$.

More generally, if $n=r_{1} r_{2} \cdots r_{s}$ where $r_{1}, r_{2},
\ldots ,  r_{s} \in \mathbb{N}$, then the indices \textit{j} and
\textit{i} can be expressed like this:

\[
j= j_{s-1}r_{1} r_{2} \cdots r_{s-1} + j_{s-2}r_{1} r_{2} \cdots
r_{s-2} + \ldots + j_{1}r_{1} + j_{0}
\]

where

\[
j_{k-1} = 0,1, \ldots , r_{k} - 1, \quad 1 \leq k \leq s
\]

and

\[
i= i_{s-1}r_{2} r_{3} \cdots r_{s} + i_{s-2}r_{3} r_{4} \cdots r_{s}
+ \ldots + i_{1}r_{s} + i_{0}
\]

where

\[
i_{k-1} = 0,1, \ldots , r_{s-(k-1)} - 1, \quad 1 \leq k \leq s
\]

Now equation (\ref{V}) can be rewritten \cite{Bergland}, by setting
$v_{i}=x_{0}(i_{s-1},i_{s-2}, \ldots , i_1 , i_0 )$ and $V_{j} = V
(j_{s-1},j_{s-2}, \ldots , j_1 , j_0 )$, as

\begin{equation*}
V (j_{s-1},j_{s-2}, \ldots , j_1 , j_0 )=\sum_{i_{0}=0}^{r_{s}-1}
\sum_{i_{1}=0}^{r_{s-1}-1} \cdots
\sum_{i_{s-1}=0}^{r_{1}-1}x_{0}(i_{s-1},i_{s-2}, \ldots , i_1 , i_0
) {\omega^{ij}}
\end{equation*}

Using the fact that $ \omega^{r_1 r_2 \cdots r_s} = \omega^n = 1$
this expression can be calculated by \textit{s} recursive equations:
\begin{equation}\label{recursive}
x_{1}(j_{0},i_{s-2},i_{s-3}, \ldots , i_1 , i_0
)=\sum_{i_{s-1}=0}^{r_{1}-1}x_{0}(i_{s-1},i_{s-2}, \ldots , i_1 ,
i_0 ) \omega^{j_0 i_{s-1} r_2 \cdots r_s}
\end{equation}
\[
x_{k}(j_{0},j_{1}, \ldots , j_{k-1}, i_{s-k-1}, \ldots , i_1 , i_0
)=
\]
\[
\sum_{i_{s-k}=0}^{r_{k}-1}x_{k-1}(j_0, j_1, \ldots , j_{k-2},
i_{s-k}, \ldots , i_1 , i_0 ) \omega^{(j_{k-1}r_1 r_2  \cdots
r_{k-1}+j_{k-2}r_1 r_2 \cdots r_{k-2} + \cdots + j_0) i_{s-k}
r_{k+1}r_{k+2} \cdots r_s}
\]
for $k= 2,3, \ldots, s-1$
\[
x_{s}(j_{0},j_{1}, \ldots , j_{s-1})=
\sum_{i_{0}=0}^{r_{s}-1}x_{s-1}(j_0, j_1, \ldots, j_{s-2}, i_0 )
\omega^{(j_{s-1}r_1 r_2 \cdots r_{s-1}+j_{s-2}r_1 r_2 \cdots r_{s-2}
+ \cdots + j_0) i_{0}}
\]

Now the final output $x_{s}(j_{0},j_{1}, \ldots , j_{s-1})= V
(j_{s-1},j_{s-2}, \ldots , j_1 , j_0 )=V_j$. \\
This algorithm will require a total of $n(r_1 + r_2 + \cdots + r_s)$
multiplications and $n(r_1 + r_2 + \cdots + r_s - s)$ additions in
$\mathbb{F}_{p^m}$. Here we did include the multiplications by
$\omega^0 = 1$.
\\ For $n=r^\nu$, the algorithm requires a total of $n \nu r =
\frac{r}{\log _{2}(r)}n \log_{2}(n)$ multiplikations. The factor
$\frac{r}{\log _{2}(r)}$ achieves its minimum for $r=3$, but $r=2$
and $r=4$ is still better because of the possibility of reducing the
numbers of multiplications using:

\begin{lem}\label{even}
Let \emph{$\omega \in \mathbb{F}_{p^m}$} be of order $n$. If
\emph{$n$} is even and \emph{$t \in \mathbb{Z}$} , then
\emph{$\omega ^{\frac{n}{2}+t} = - \omega^{t}$}
\end{lem}

\begin{proof}
The order of $\omega$ is $n \mid p^m -1$. Hence $\omega ^n = 1$. So
$0=\omega^n-1=(\omega^{\frac{n}{2}}-1)(\omega^{\frac{n}{2}}+1)$.
Since ord($\omega$)$=n$ then $\omega^{\frac{n}{2}} \neq 1$ and hence
$\omega^{\frac{n}{2}} +1 =0$
\end{proof}

For $n=2^{\mu}$ use of the lemma will reduce the number of
multiplications in $\mathbb{F}_{2^m}$ by $50 \%$. This is for $s
\geq 3$ caused by the possibility of rearranging the recursive
equations (\ref{recursive}) in a slightly different way (in
principle due to \cite{Gentleman}):

\pagebreak

\begin{equation}\label{twiddle}
y_{1}(j_{0},i_{s-2},i_{s-3}, \ldots , i_1 , i_0
)=(\sum_{i_{s-1}=0}^{r_{1}-1}x_{0}(i_{s-1},i_{s-2}, \ldots , i_1 ,
i_0 ) \omega^{j_0 i_{s-1}(n/r_1)})\omega^{j_0 i_{s-2}r_3 \cdots r_s}
\end{equation}
\[
y_{k}(j_{0},j_{1}, \ldots , j_{k-1}, i_{s-k-1}, \ldots , i_1 , i_0
)=
\]
\[
(\sum_{i_{s-k}=0}^{r_{k}-1}y_{k-1}(j_0, j_1, \ldots , j_{k-2},
i_{s-k}, \ldots , i_1 , i_0 ) \omega^{(j_{k-1}i_{s-k}(n/r_{k}))})
\omega^{(j_{k-1} r_1 r_2 \cdots r_{k-1} + \cdots + j_1 r_1 +
j_0)i_{s-k-1}r_{k+2} \cdots r_s}
\]
for $k= 2,3, \ldots, s-1$ when $s \geq 4$, else go to the next
equation.

\[
y_{s-1}(j_{0},j_{1}, \ldots , j_{s-2}, i_{0})=
\]
\[
(\sum_{i_{1}=0}^{r_{s-1}-1}y_{s-2}(j_0, j_1, \ldots , j_{s-3}, i_1 ,
i_0 ) \omega^{(j_{s-2}i_{1}(n/r_{s-1}))}) \omega^{(j_{s-2} r_1 r_2
\cdots r_{s-2} + \cdots + j_1 r_1 + j_0)i_{0}}
\]

\[
y_{s}(j_{0},j_{1}, \ldots , j_{s-1})=
\sum_{i_{0}=0}^{r_{s}-1}y_{s-1}(j_0, j_1, \ldots, j_{s-2}, i_0 )
\omega^{(j_{s-1}i_{0}(n/r_s))}
\]
Here the final output $y_{s}(j_{0},j_{1}, \ldots , j_{s-1})= V
(j_{s-1},j_{s-2}, \ldots , j_1 , j_0 )=V_j$. \\

For $\ell = 1,2, \cdots s$,  a  $r_{\ell}$ -- point Fourier
Transform is included in step number $\ell$ of the algorithm. Among
these, each two point Fourier Transform does not require any
multiplication because
$\omega^0 = 1$ and $\omega^{\frac{n}{2}}=-\omega$. \\

Within the original Cooley -- Tukey algorithm, which is executed
over the field of complex numbers, it is possible to do additional
tricks by looking at the real and imaginary part of a number. These
tricks can not be transferred to a finite field. \\ \\
The overall conclusion must be that the algorithm sketched above
will be relatively most efficient if the total number of points
$n=r_{1}^{\nu_{1}}r_{2}^{\nu_{2}} \cdots r_{u}^{\nu_{u}}$ is
factored in factors as small as possible.

\section{Concrete suggestions}
Versions of the Cooley -- Tukey algorithm are not very efficient
over binary fields $\mathbb{F}_{2^m}$ where $m \in \mathbb{N}$. In
most of these cases the order $2^m - 1$ of the multiplicative group
is not highly composite. For instance the order of the fields
examined in the recent article \cite{Truong}, \cite{Truong er} are
$2^8, 2^9$ and $2^{10}$ where $2^8-1 = 3 \times 5 \times 17, \quad
2^9 - 1 = 7 \times 73$ and $2^{10} -1 = 3 \times 11 \times 31$ It
could also be mentioned that $2^7-1$ is a prime. The algorithm
presented in \cite{Truong}, \cite{Truong er} is a Prime Factor
Algorithm which as such takes advantage of the fact that all the
prime factors in $2^m -1 $ are coprimes for $8 \leq n \leq 10 $.
This will also be the case for a great deal of the numbers $2^m -1$
for bigger $m$, but some of the prime factors tend to be bigger too.
For example $2^{15} - 1 = 7 \times 31 \times 151, \quad 2^{16} - 1 =
3 \times 5 \times 17 \times 257, \quad 2^{17} -1$ is prime, $2^{18}
-1 = 3^3 \times 7 \times 19 \times 73$ and  $2^{19}
-1$ is prime. \\
It is obvious, as mentioned, that versions of the Cooley -- Tukey
algorithm will not be very efficient in finite fields like these. It
will in these cases be much more efficient to avoid using all $m$
bits in its full content, and use algorithm (\ref{recursive}) or
(\ref{twiddle}) over a prime field instead. Here are two examples: \\

Instead of using 17 bits to create the field $\mathbb{F}_{2^{17}}$
with a multiplicative af order $131071$ which is a prime, then use
18 bits to create the prime field $\mathbb{F}_{147457}$ which
multiplicative group is of the order $2^{14} \times 3^2$. \\
Or instead of using $19$ bits to create the field
$\mathbb{F}_{2^{19}}$ with a multiplicative group of the order
$524287$ which is also a prime, then use an extra bit to create the
field $\mathbb{F}_{786433}$ which multiplicative group is of the
order $2^{18} \times 3$. \\
The orders of the multiplicative groups of the prime fields given in
the two examples are highly composite, and the algorithm
(\ref{twiddle}), which is based on the Cooley -- Tukey algorithm,
will be very efficient here: In the case $\mathbb{F}_{147457}$, DFT
uses $(2^{14} \times 3^2)^2 \approx 2 \times 10^{10}$
multiplications and the FFT (\ref{twiddle}) suggested here will
require $2^{14} \times 3^2 \times (14 \times 1 + 2 \times 3) \approx
3 \times 10^6$ multiplications, which is $\frac{2^{14} \times
3^2}{14 \times 1 + 2 \times 3} \approx 7 \times 10^3$ times faster
than the DFT. Here we have used lemma \ref{even} to reduce the
number of multiplications. The multiplication in it self is also
easy: Just multiplication modulo the prime, which in the example is
$147457$. The estimate is roughly the same as regards the additions:
The DFT over $\mathbb{F}_{147457}$ requires $147457 \times (147457 -
1) \approx 2 \times 10^{10}$ additions and our FFT (\ref{twiddle})
requires $147457 \times (14 \times 2 + 2 \times 3 -(14+2)) = 3
\times 10^6$ additions, which is $\frac{147457-1}{14 \times 2 + 2
\times 3 -(14+2)} \approx 8 \times 10^3$ times better.
\\
In the second example $\mathbb{F}_{786433}$, our FFT (\ref{twiddle})
will perform the multiplications $\frac{2^{18} \times 3}{18 \times 1
+3 } \approx 4 \times 10^4$ times faster than the DFT. And the
additions will similarly be performed $\frac{786433-1}{18 \times 2 +
3 -(18+1)} \approx 4 \times 10^4$ faster than the DFT.

A FFT calculated in for instance 1 second, would then take roughly
10 hours as a DFT.

\section{The elements of order n}
In our FFT (\ref{twiddle}) over $\mathbb{F}_p$ an element $\omega$
of order $n \mid p-1$ appears. Usually we will choose $n=p-1$, and
then $\omega$ will be a generator of $\mathbb{F}_p$. Such a
generator will normally be easy to find: according to Lagrange's
theorem in a finite group the order of any element will be a divisor
in the order of the group. Therefore an element $a$ is a generator
of the multiplicative subgroup of $\mathbb{F}_p$ with $n$ elements
iff
\[
a^{n/r} \neq 1 \ \mathrm{mod} \ p \quad \mathrm{for \ every \ prime
\ factor}\ r  \ \mathrm{of} \ n.
\]

A probabilistic algorithm to determine the smallest possible
generator of the multiplicative subgroup of $\mathbb{F}_{p}$ with
$n$ elements goes like this:

\begin{alg}
      \ \\
\textbf{Input}: $n \mid p-1$
\begin{enumerate}
\item Prime factorize $n$
\item Choose the smallest integer $a$ from the set $\{2,3, \ldots ,
n\}$
\item For every primefactor $r$ of n calculate $a^{n/r}$.
\item If this quantity is different from $1$ for all primefactors
$r$ of $n$, then stop. Else repeat step $2$ and $3$ for the lowest
values of $a \in \{2, 3, \ldots , n \} $  until this happens. Then
stop.
\end{enumerate}
\textbf{output:} The  last  value  of  $a$.
\end{alg}
\textbf{Comments on the algorithm:} For practical puposes $n <
2^{30}$ , and then the prime factorization of $n$ will not be
computationally difficult. The algorithm will allways find a
generator $\omega$, as we know that it exists. If the prime
factorization of $n$ is $n = p^{\nu_1}_1 p^{\nu_2}_2 \cdots
p^{\nu_u}_u$ then the numbers of generators of the multiplicative
subgroup of $\mathbb{F}_p$ will be
$\varphi(n)=n(1-\frac{1}{p_1})(1-\frac{1}{p_2}) \ldots
(1-\frac{1}{p_u})$ (see e.g. \cite{kob}). Hence the possibility of a
random $a \in \{2,3, \ldots , n\}$ being a generator $\omega$ equals
$(1-\frac{1}{p_1})(1-\frac{1}{p_2}) \ldots (1-\frac{1}{p_u})$. In
both the earlier mentioned examples $\mathbb{F}_{147457}$ and
$\mathbb{F}_{786433}$ this probability is
$(1-\frac{1}{2})(1-\frac{1}{3})= \frac{1}{3}$

\section{A list of suitable choices of primes p}

At the end of this article I will print a list of primes
$2^{16}<p<2^{21}$ where the only primefactors in $p-1$ are $2$ and
$3$. For these primes the FFT algorithms (\ref{recursive}) and
(\ref{twiddle}) treated here will be especially efficient. For
$n=p-1=2^{\nu _1} 3^{\nu _2}$ the number of multiplications in
algorithm (\ref{recursive}) will be $(p-1)(2 \nu _1 + 3 \nu _2)$
which can be reduced to $(p-1)(\nu _1 + 3 \nu _2)$ using algorithm
(\ref{twiddle}). The number of additions will for both FFT
algorithms be $( p -1)(2 \nu _1 + 3 \nu _2 - (\nu_1 + \nu_2)) =
(p-1)(\nu_1 + 2 \nu_2)$.

As we see, the table starts with the biggest Fermat -- number
$2^{2^n} + 1$ known to be a prime. For nearly half of the shown
numbers $p-1$ a generator $ \omega $ of $\mathbb{F}_p$ is $ 5 = 2^2
+ 1$, a nice number to multiply with in base $2$. \\
The factoring was implemented with the math -- program Maple on my
mobile PC. \\

\begin{center}
\begin{tabular}{|r|c|r|}
  \hline
  Prime $p$ & Factorization of $p-1$ & Generator $\omega$ \\
  \hline
    65537& $2^{16}$ & 3 \\
  139969 & $2^6 \times 3^7$ & 13 \\
  147457 & $2^{14} \times 3^2$ & 10 \\
  209953 & $2^5 \times 3^8$ & 10 \\
  331777 & $2^{12} \times 3^4$ & 5 \\
  472393 & $2^3 \times 3^{10}$ & 5 \\
  629857 & $2^5 \times 3^9$ & 5 \\
  746497 & $2^{10} \times 3^6$ & 5 \\
  786433 & $2^{18} \times 3$ & 10 \\
  839809 & $2^7 \times 3^8$ & 7 \\
  995329 & $2^{12} \times 3^5$ & 7 \\
  1179649 & $2^{17} \times 3^2$ & 19 \\
  1492993 & $2^{11} \times 3^6$ & 7 \\
  1769473 & $2^{16} \times 3^3$ & 5 \\
  1990657 & $2^{13} \times 3^5$ & 5 \\
  \hline
\end{tabular}

\end{center}

\ \\
\ \\

\ \\

\large{\textbf{Perspective:}}
\normalsize \\
The number of simple factorizations like those above seems not to
stop when the primes grow even bigger. Here are two examples:

For the prime $p = 113246209$ the factorization of $p-1$ is
$2^{22}\times 3^3$.

For the prime $p = 725594113$ the factorization of $p-1$ is
$2^{12}\times 3^{11}$.

\pagebreak

\end{document}